\documentclass[10pt]{amsart}
\usepackage{amsmath}
\usepackage{amscd}
\usepackage{amssymb}
\usepackage{amsthm}
\RequirePackage{filecontents}
\usepackage[numbers]{natbib}  
\usepackage[draft]{hyperref}

\theoremstyle{plain}
\newtheorem{prop}{Proposition}
\newtheorem{lem}[prop]{Lemma}
\newtheorem{conj}[prop]{Conjecture}
\theoremstyle{definition}
\newtheorem{defn}[prop]{Definition}
\newtheorem{rem}[prop]{Remark}
\newtheorem{ex}[prop]{Example}

\newcommand\numberthis{\addtocounter{equation}{1}\tag{\theequation}}

\title{Vector-valued Eisenstein series of small weight}
\author{Brandon Williams }

\address{Department of Mathematics \\ University of California \\ Berkeley, CA 94720}

\email{btw@math.berkeley.edu}

\begin{document}

\nocite{*}

\maketitle

\begin{abstract} We study the (mock) Eisenstein series $E_k$ of weight $k \in \{1,3/2,2\}$ for the Weil representation on an even lattice, defined as the result of Bruinier and Kuss's coefficient formula for the Eisenstein series naively evaluated at $k$. We describe the transformation law of $E_k$ in general. Most of this note is dedicated to collecting examples where the coefficients of $E_k$ contain interesting arithmetic information. Finally we make a few remarks about the case $k=1/2$.
\end{abstract}

\begin{small} \keywordsname{: Modular forms; Mock modular forms; Eisenstein series; Weil representation} \end{small}

\baselineskip=13pt

\section{Introduction}

In \cite{BK}, Bruinier and Kuss give an expression for the Fourier coefficients of the Eisenstein series $E_k$ of weight $k \ge 5/2$ for the Weil representation attached to a discriminant form. These coefficients involve special values of $L$-functions and zero counts of polynomials modulo prime powers, and they also make sense for $k \in \{1,3/2,2\}$. Unfortunately, the $q$-series $E_k$ obtained in this way often fail to be modular forms. In particular, in weight $k = 3/2$ and $k = 2$, the Eisenstein series may be a mock modular form that requires a real-analytic correction in order to transform as a modular form. Many examples of this phenomenon of the Eisenstein series are well-known (although perhaps less familiar in a vector-valued setting). We will list a few examples of this:

\begin{ex} The Eisenstein series of weight $2$ for a unimodular lattice $\Lambda$ is the quasimodular form $$E_2(\tau) = 1 - 24\sum_{n=1}^{\infty} \sigma_1(n) q^n = 1 - 24q - 72q^2 - 96q^3 - 168q^4 - ...$$ where $\sigma_1(n) = \sum_{d | n} d$, which transforms under the modular group by $$E_2\Big( \frac{a \tau + b}{c \tau + d} \Big) = (c \tau + d)^2 E_2(\tau) + \frac{6}{\pi i} c (c \tau + d).$$
\end{ex}

\begin{ex} The Eisenstein series of weight $3/2$ for the quadratic form $q_2(x) = x^2$ is essentially Zagier's mock Eisenstein series: $$E_{3/2}(\tau) = \Big( 1 - 6q - 12q^2 - 16q^3 - ... \Big) \mathfrak{e}_0 + \Big( -4q^{3/4} - 12q^{7/4} - 12q^{11/4} - ... \Big) \mathfrak{e}_{1/2},$$ in which the coefficient of $q^{n/4} \mathfrak{e}_{n/2}$ is $-12$ times the Hurwitz class number $H(n).$ It transforms under the modular group by $$E_{3/2} \Big( \frac{a \tau + b}{c \tau + d} \Big) = (c \tau + d)^{3/2} \rho^*\Big( \begin{pmatrix} a & b \\ c & d \end{pmatrix} \Big) \Big[ E_{3/2}(\tau) - \frac{3}{\pi} \sqrt{\frac{i}{2}} \int_{d/c}^{i \infty} (\tau + t)^{-3/2} \vartheta(t) \, \mathrm{d}t \Big],$$ where $\vartheta$ is the theta series $$\vartheta(\tau) = \sum_{n \in \mathbb{Z}} q^{n^2 / 4} \mathfrak{e}_{n/2}.$$
\end{ex}

\begin{ex} In the Eisenstein series of weight $3/2$ for the quadratic form $q_3(x) = -12x^2,$ the components of $\mathfrak{e}_{1/12}$, $\mathfrak{e}_{5/12},$ $\mathfrak{e}_{7/12}$ and $\mathfrak{e}_{11/12}$ are $$\Big( -3q^{23/24} - 5q^{47/24} - 7q^{71/24} - 8q^{95/24} - 10q^{119/24} - 10q^{143/24} - ... \Big) \mathfrak{e}_{\gamma}$$ for $\gamma \in \{1/12,5/12,7/12,11/12\}.$ We verified by computer that the coefficient of $q^{n - 1/24}$ above is $(-1)$ times the degree of the $n$-th partition class polynomial considered by Bruinier and Ono \cite{BO} for $1 \le n \le 750$, which is not surprising in view of example $2$ since this degree also counts equivalence classes of certain binary quadratic forms. This Eisenstein series is not a modular form.
\end{ex}

\begin{ex} The Eisenstein series of weight $3/2$ for the quadratic form $q_4(x,y,z) = x^2 + y^2 - z^2$ is a mock modular form that is related to the functions considered by Bringmann and Lovejoy \cite{BL} in their work on overpartitions. More specifically, the component of $\mathfrak{e}_{(0,0,0)}$ in $E_{3/2}$ is $$1 - 2q - 4q^2 - 8q^3 - 10q^4 - ... = 1 - \sum_{n=1}^{\infty} |\overline{\alpha}(n)| q^n,$$ where $\overline{\alpha}(n)$ is the difference between the number of even-rank and odd-rank overpartitions of $n$. Similarly, the $M2$-rank differences considered in \cite{BL} appear to occur in the Eisenstein series of weight $3/2$ for the quadratic form $q_5(x,y,z) = 2x^2 + 2y^2 - z^2$, whose $\mathfrak{e}_{(0,0,0)}$-component is $$1 - 2q - 4q^2 - 2q^4 - 8q^5 - 8q^6 - 8q^7 - ...$$
\end{ex}

\begin{ex} Unlike the previous examples, the Eisenstein series of weight $3/2$ for the quadratic form $q_6(x,y,z) = -x^2 -y^2 -z^2$ is a true modular form; in fact, it is the theta series for the cubic lattice and the Fourier coefficients of its $\mathfrak{e}_{(0,0,0)}$-component count the representations of integers as sums of three squares.
\end{ex}

Among negative-definite lattices of small dimension there are lots of examples where the Eisenstein series equals the theta series. (Note that we find theta series for negative-definite lattices instead of positive-definite because we consider the dual Weil representation $\rho^*$.) When the lattice is even-dimensional this immediately leads to formulas for representation numbers in terms of twisted divisor sums. These formulas are of course well-known but the vector-valued derivations of these formulas seem more natural than the usual derivation as identities among scalar-valued forms of higher level. We give several examples of this throughout the note. \\

In the last section we make some remarks about the case $k=1/2$, where the formula of \cite{BK} no longer makes sense and so the methods of this note break down. \\

\textbf{Acknowledgments:} I thank Kathrin Bringmann for discussing the examples involving overpartition rank differences with me.

\section{Background}

In this section we review some facts about the metaplectic group and vector-valued modular forms, as well as Dirichlet $L$-functions, which will be useful later. \\

Recall that the \textbf{metaplectic group} $Mp_2(\mathbb{Z})$ is the double cover of $SL_2(\mathbb{Z})$ consisting of pairs $(M,\phi),$ where $M = \begin{pmatrix} a & b \\ c & d \end{pmatrix} \in SL_2(\mathbb{Z})$, and $\phi$ is a branch of $\sqrt{c \tau + d}$ on the upper half-plane $$\mathbb{H} = \{\tau = x+iy \in \mathbb{C}: \, y > 0\}.$$ We will usually omit $\phi$. $Mp_2(\mathbb{Z})$ is generated by the elements $$T = \Big( \begin{pmatrix} 1 & 1 \\ 0 & 1 \end{pmatrix}, 1 \Big), \; \; S = \Big( \begin{pmatrix} 0 & -1 \\ 1 & 0 \end{pmatrix}, \sqrt{\tau} \Big),$$ with defining relations $S^8 = I$ and $S^2 = (ST)^3.$ \\

Let $\Lambda$ be a lattice (which we can always take as $\Lambda = \mathbb{Z}^e$ for some $e \in \mathbb{N}$) with an even quadratic form $q : \Lambda \rightarrow \mathbb{Z}$, and let $$\Lambda' = \{v \in \Lambda \otimes \mathbb{Q}^e: \, \langle v,w \rangle \in \mathbb{Z}\, \mathrm{for} \, \mathrm{all} \, w \in \Lambda\}$$ be the dual lattice. We denote by $\mathfrak{e}_{\gamma}$, $\gamma \in \Lambda'/\Lambda$ the natural basis of the group algebra $\mathbb{C}[\Lambda'/\Lambda]$. The \textbf{Weil representation} of $Mp_2(\mathbb{Z})$ attached to $\Lambda$ is the map $$\rho : Mp_2(\mathbb{Z}) \longrightarrow \mathrm{Aut}\, \mathbb{C}[\Lambda'/\Lambda]$$ defined by $$\rho(T) \mathfrak{e}_{\gamma} = \mathbf{e}\Big( q(\gamma) \Big) \mathfrak{e}_{\gamma}, \; \; \rho(S) \mathfrak{e}_{\gamma} = \frac{\sqrt{i}^{b^- - b^+}}{\sqrt{|\Lambda'/\Lambda|}} \sum_{\beta \in \Lambda'/\Lambda} \mathbf{e}\Big( -\langle \gamma, \beta \rangle \Big) \mathfrak{e}_{\beta}.$$ In particular, $$\rho(Z) \mathfrak{e}_{\gamma} = i^{b^- - b^+} \mathfrak{e}_{-\gamma}, \; \; \mathrm{where} \; \; Z = (-I,i) = S^2 = (ST)^3.$$ Here we use $\mathbf{e}(x)$ to denote $e^{2\pi i x},$ and $(b^+,b^-)$ is the signature of $\Lambda.$ \\ We will usually consider the dual representation $\rho^*$ of $\rho$ (which also occurs as the Weil representation itself, for the lattice $\Lambda$ and quadratic form $-q$). \\

A \textbf{modular form of weight $k$} for $\rho^*$ is a holomorphic function $f : \mathbb{H} \rightarrow \mathbb{C}[\Lambda'/\Lambda]$ with the properties: \\ (i) $f$ transforms under the action of $Mp_2(\mathbb{Z})$ by $$f(M \cdot \tau) = (c \tau +d)^k \rho^*(M) f(\tau), \; \; M = \begin{pmatrix} a & b \\ c & d \end{pmatrix} \in Mp_2(\mathbb{Z}),$$ where the branch of $(c \tau + d)^k$ is prescribed by $M$ as an element of $Mp_2(\mathbb{Z})$ if $k$ is half-integer; \\ (ii) $f$ is holomorphic in $\infty$. This means that in the Fourier expansion $$f(\tau) = \sum_{\gamma \in \Lambda'/\Lambda} \sum_{n \in \mathbb{Z} - q(\gamma)} c(n,\gamma) q^n \mathfrak{e}_{\gamma},$$ all coefficients $c(n,\gamma)$ are zero for $n < 0.$ \\

If $N$ is the smallest natural number such that $N \langle \gamma, \beta \rangle$ and $Nq(\gamma) \in \mathbb{Z}$ for all $\beta,\gamma \in \Lambda'/\Lambda,$ then $\rho^*$ factors through $SL_2(\mathbb{Z}/N\mathbb{Z})$ if $e = \mathrm{dim}\, \Lambda$ is even, and through a double cover of $SL_2(\mathbb{Z}/N\mathbb{Z})$ if $e$ is odd. This implies in particular that the component functions $f_{\gamma}$ of $f$ are scalar modular forms of level $N$. \\

When studying the weight $3/2$ Eisenstein series we will consider \textbf{harmonic weak Maass forms}, which have the same transformation behavior as modular forms but for which the holomorphy assumption is weakened to real-analyticity and the weight-$k$ Laplace equation $\Delta f(\tau) = 2iky \frac{\partial}{\partial \overline{\tau}} f(\tau),$ where $\Delta = y^2 ( \frac{\partial^2}{\partial x^2} + \frac{\partial^2}{\partial y^2})$ is the hyperbolic Laplacian on $\mathbb{H}$. Harmonic weak Maass forms are also required to satisfy a growth condition of the form $|f(\tau)| < C e^{Ny}$ at $\infty$. We refer to \cite{BF} and \cite{O} for details. \\

The weights of modular forms are restricted due to $$f(\tau) = f(Z \cdot \tau) = i^{2k} \rho^*(Z) f(\tau) = i^{2k + b^+ - b^-} \sum_{\gamma \in \Lambda'/ \Lambda} f_{-\gamma}(\tau) \mathfrak{e}_{\gamma}.$$ In particular, if $2k + b^+ - b^-$ is not an even integer, then there are no nonzero modular forms. In the case $2k + b^+ - b^- \equiv 2 \, (4)$ (which seems to be of less interest), the components satisfy $f_{\gamma} = -f_{-\gamma}$ and in particular the $\mathfrak{e}_0$-component of $f$ must be zero. We will consider only the case $2k + b^+ - b^- \equiv 0 \, (4)$ as we are interested in Eisenstein series with constant term $1 \cdot \mathfrak{e}_0$. \\

\begin{rem} There is an involution $\sim$ of the metaplectic group given on the standard generators by $$\tilde S = S^{-1}, \; \; \tilde T = T^{-1},$$ which is well-defined because $$(\tilde S \tilde T)^3 = S^{-1} (ST)^{-3} S = S^{-1} S^{-2} S = \tilde S^2$$ and $\tilde S^8 = I.$ On matrices it is given by $$\widetilde{\begin{pmatrix} a & b \\ c & d \end{pmatrix}} = \begin{pmatrix} a & -b \\ -c & d \end{pmatrix},$$ and it acts on the branches of square roots by $\tilde \phi(\tau) = \overline{\phi(-\overline{\tau})},$ where $\phi(\tau)^2 = c \tau + d.$ One can check on the generators $S,T$ that this intertwines the Weil representation $\rho$ and its dual $\rho^*$ in the sense that $$\rho(\tilde M) = \rho^*(M) = \overline{\rho(M)}, \; \; M \in Mp_2(\mathbb{Z}).$$
\end{rem}

\begin{rem} At many points in this note we will need to consider the $L$-function $$L(s,\chi_D) = \sum_{n=1}^{\infty} \chi_D(n) n^{-s}$$ attached to the Dirichlet character mod $|D|$, $$\chi_D(n) = \Big( \frac{D}{n} \Big),$$ where $D$ is a discriminant (i.e. $D \equiv 0,1 \, \bmod \, 4$). In particular, we recall the following properties of Dirichlet $L$-functions. \\ (i) Let $\chi$ be a Dirichlet character. Then $L(s,\chi)$ converges absolutely in some half-plane $\mathrm{Re}[s] > s_0$ and is given by an Euler product $$L(s,\chi) = \prod_{p \, \mathrm{prime}} (1 - \chi(p) p^{-s})^{-1}$$ there. \\ (ii) $L(s,\chi)$ has a meromorphic extension to all $\mathbb{C}$ and satisfies the functional equation $$\Gamma(s) \cos\Big( \frac{\pi (s - \delta)}{2} \Big) L(s,\chi) = \frac{\tau(\chi)}{2i^{\delta}} (2\pi / f)^s L(1-s, \overline{\chi}),$$ where $f$ is the conductor of $\chi$, $\tau(\chi) = \sum_{a=1}^f \chi(a) e^{2\pi i a / f}$ is the Gauss sum of $\chi$, and $$\delta = \begin{cases} 1: & \chi(-1) = -1; \\ 0: & \chi(-1) = 1. \end{cases}$$ (iii) $L(s,\chi)$ is never zero at $s = 1$, and is holomorphic there unless $\chi$ is a trivial character, in which case it has a simple pole. \\ (iv) The values $L(1-n,\chi)$, $n \in \mathbb{N}$ are rational numbers, given by $$L(1-n,\chi) = -\frac{B_{n,\chi}}{n},$$ where $B_{n,\chi} \in \mathbb{Q}$ is a generalized Bernoulli number. \\ We refer to section 4 of \cite{Wa} for these and other results on Dirichlet $L$-functions.
\end{rem}

\section{The real-analytic Eisenstein series}

Fix an even lattice $\Lambda$ and let $\rho^*$ be the dual Weil representation on $\mathbb{C}[\Lambda'/\Lambda].$

\begin{defn} The \textbf{real-analytic Eisenstein series} of weight $k$ is $$E_k^*(\tau,s) = \sum_{M \in \tilde \Gamma_{\infty} \backslash \Gamma} (y^s \mathfrak{e}_0) |_k M = \frac{y^s}{2} \sum_{c,d} (c \tau + d)^{-k} |c \tau + d|^{-2s} \rho^*(M)^{-1} \mathfrak{e}_0.$$ Here, $(c,d)$ runs through all pairs of coprime integers and $M$ is any element $\begin{pmatrix} a & b \\ c & d \end{pmatrix} \in Mp_2(\mathbb{Z})$ with bottom row $(c,d)$; and the branch of $(c \tau + d)^{-k}$ is determined by $M$ as an element of $Mp_2(\mathbb{Z})$ as usual.
\end{defn}

This series converges absolutely and locally uniformly in the half-plane $\mathrm{Re}[s] > 1 - k/2$ and defines a holomorphic function in $s$. For fixed $s$, it transforms under the metaplectic group by $$E_k^*\Big( M \cdot \tau,s \Big) = (c \tau + d)^k \rho^*(M) E_k^*(\tau,s)$$ for any $M = \begin{pmatrix} a & b \\ c & d \end{pmatrix} \in Mp_2(\mathbb{Z}).$ These series were considered by Bruinier and K\"uhn \cite{BK2} in weight $k \ge 2$ who also give expressions for their Fourier expansions. (More generally they consider the series obtained after replacing $\mathfrak{e}_0$ with $\mathfrak{e}_{\beta}$ for an element $\beta \in \Lambda'/\Lambda$ with $q(\beta) \in \mathbb{Z}$. We do not do this because it seems to make the formulas below considerably more complicated, and because for many discriminant forms $\Lambda'/\Lambda$ one can obtain the real-analytic Eisenstein series associated to any $\beta$ from the $E_k^*(\tau,s)$ above by a simple ``averaging" argument. See for example the appendix of \cite{W}.) \\

The series $E_k^*(\tau,s)$ can be analytically extended beyond the half-plane $\mathrm{Re}[s] > 1 - k/2$. We will focus here on weights $k \in \{1,3/2,2\}$, in which the Fourier series is enough to give an explicit analytic continuation to $s=0$. First we work out an expression for the Fourier series (in particular, our result below differs in appearance from \cite{BK2} because we use a different computation of the Euler factors). Writing $$E_k^*(\tau,s) = \mathfrak{e}_0 + \sum_{\gamma \in \Lambda'/\Lambda} \sum_{n \in \mathbb{Z}-q(\gamma)} c(n,\gamma,s,y) q^n \mathfrak{e}_{\gamma},$$ a computation analogous to section 1.2.3 of \cite{Br} using the exact formula for the coefficients $\rho(M)_{0,\gamma}$ of the Weil representation cited there shows that \begin{align*} c(n,\gamma,s,y) &= \frac{y^s}{2} \sum_{c \ne 0} \sum_{d \in (\mathbb{Z}/c\mathbb{Z})^{\times}} \rho(M)_{0,\gamma} \int_{-\infty + iy}^{\infty + iy} (c \tau + d)^{-k} |c \tau + d|^{-2s} \mathbf{e}(-n\tau) \, \mathrm{d}x \\ &= y^s \sum_{c=1}^{\infty} \sum_{d \in (\mathbb{Z}/c\mathbb{Z})^{\times}} \rho(M)_{0,\gamma} c^{-k-2s} \mathbf{e}\Big( \frac{nd}{c} \Big) \int_{-\infty + iy}^{\infty + iy} \tau^{-k} |\tau|^{-2s} \mathbf{e}(-n\tau) \, \mathrm{d}x \\ &= \numberthis \frac{\sqrt{i}^{b^- - b^+}}{\sqrt{|\Lambda'/\Lambda|}} \tilde L(n,\gamma,k + e/2 + 2s) I(k,y,n,s), \end{align*} where $M$ is any element of $Mp_2(\mathbb{Z})$ whose bottom row is $(c,d).$ Here, $\tilde L(n,\gamma,s)$ is the $L$-series 

\begin{align*} \tilde L(n,\gamma,s) &= \sum_{c=1}^{\infty} c^{-s+e/2} \sum_{d \in (\mathbb{Z}/c\mathbb{Z})^{\times}} \rho(M)_{0,\gamma} \mathbf{e}\Big( \frac{nd}{c} \Big) \\ &= \sum_{c=1}^{\infty} c^{-s} \sum_{\substack{v \in \Lambda/c \Lambda \\ d \in (\mathbb{Z}/c \mathbb{Z})^{\times}}} \mathbf{e}\Big( \frac{a q(v) - \langle \gamma,v \rangle + d q(\gamma) - nd}{c} \Big) \\ &= \sum_{c=1}^{\infty} c^{-s} \sum_{a | c} \Big[  \mu(c/a) a (c/a)^e \cdot \# \Big\{ v \in \Lambda/a \Lambda: \, q(v-\gamma) + n \equiv 0 \, (\bmod \,a) \Big\} \Big] \\ &= \zeta(s-e)^{-1} L(n,\gamma,s-1), \end{align*} where $L(n,\gamma,s)$ is $$L(n,\gamma,s) = \sum_{a=1}^{\infty} a^{-s} \mathbf{N}(a) = \prod_{p\, \mathrm{prime}} \Big( \sum_{\nu=0}^{\infty} p^{-\nu s} \mathbf{N}(p^{\nu}) \Big) = \prod_{p \, \mathrm{prime}} L_p(n,\gamma,s),$$ and $\mathbf{N}(p^{\nu})$ is the number of zeros $v \in \Lambda/p^{\nu} \Lambda$ of the quadratic polynomial $q(v-\gamma) + n$; and $I(k,y,n,s)$ is the integral \begin{align*} I(k,y,n,s) &=y^s \int_{-\infty + iy}^{\infty + iy} \tau^{-k} |\tau|^{-2s} \mathbf{e}(-n\tau) \, \mathrm{d}x \\ &= y^{1-k-s} e^{2\pi n y} \int_{-\infty}^{\infty} (t+i)^{-k} (t^2 + 1)^{-s} \mathbf{e}(-nyt) \, \mathrm{d}t, \; \; \tau = y(t+i). \end{align*}

\begin{rem} Both the $L$-series term $\tilde L(n,\gamma,s)$ and the integral term $I(k,y,n,s)$ of (1) have meromorphic continuations to all $s \in \mathbb{C}.$ First we remark that the integral $I(k,y,n,s)$ was considered by Gross and Zagier \cite{GZ}, section IV.3., where it was shown that for $n \ne 0$, $I(k,y,n,s)$ is a finite linear combination of $K$-Bessel functions (we will not need the exact expression) and its value at $s=0$ is given by \begin{equation} I(k,y,n,0) = \begin{cases} 0: & n < 0; \\ (-2\pi i)^k n^{k-1} \frac{1}{\Gamma(k)}: & n > 0; \end{cases} \end{equation} if $n \ne 0$; and when $n = 0$, \begin{equation} I(k,y,0,s) = \pi (-i)^k 2^{2-k-2s} y^{1 - k - s} \frac{\Gamma(2s + k - 1)}{\Gamma(s) \Gamma(s+k)}. \end{equation} In particular, the zero value of the latter expression is $$I(k,y,0,0) = \begin{cases} 0: & k \ne 1; \\ -i \pi: & k = 1. \end{cases}$$

The Euler factors $L_p(n,\gamma,s) = \sum_{\nu=0}^{\infty} p^{-\nu s} \mathbf{N}(p^{\nu})$ are known to be rational functions in $p^{-s}$ that can be calculated using the methods of \cite{CKW} (see also section 6 of \cite{W} as well as the appendix, where the result of the case $p=2$ was worked out). For generic primes (primes $p \ne 2$ that do not divide $|\Lambda'/\Lambda|,$ or the numerator or denominator of $n$ if $n \ne 0$) the result is that $$L_p(n,\gamma,s) = \begin{cases} \frac{1}{1 - p^{e-1-s}} \Big[ 1 - \Big( \frac{D}{p} \Big) p^{e/2 - s} \Big]: & n \ne 0; \\ \\ \frac{1 - \Big( \frac{D'}{p} \Big) p^{e/2 - 1 - s}}{(1 - p^{e-1-s}) \Big[ 1 - \Big( \frac{D'}{p} \Big) p^{e/2 - s} \Big]}: & n = 0; \end{cases}$$ if $e$ is even and $$L_p(n,\gamma,s) = \begin{cases} \frac{1}{1 - p^{e-1-s}} \Big[ 1 + \Big( \frac{\mathcal{D}'}{p} \Big) p^{(e-1)/2 - s} \Big]: & n \ne 0; \\ \\ \frac{1 - p^{e-1-2s}}{(1 - p^{e-1-s})(1 - p^{e-2s})}: & n = 0; \end{cases}$$ if $e$ is odd. Here, $D'$ and $\mathcal{D}'$ are defined by $$D' = (-1)^{k} |\Lambda'/\Lambda| \; \; \mathrm{and} \; \; \mathcal{D}' = 2nd_{\gamma}^2 (-1)^{k - 1/2} |\Lambda'/\Lambda|.$$ 

In particular, if we define $D = D' \cdot \prod_{\mathrm{bad}\, p} p^2$ and $\mathcal{D} = \mathcal{D}' \cdot \prod_{\mathrm{bad}\, p} p^2$, where the bad primes are $2$ and any prime dividing $|\Lambda'/\Lambda|$ or $n$, then we get the meromorphic continuations $$\tilde L(n,\gamma,s) = \begin{cases} \frac{1}{L(s - e/2, \chi_D)} \prod_{\mathrm{bad}\, p} (1 - p^{e-s}) L_p(n,\gamma,s-1): & n \ne 0; \\ \\ \frac{L(s-1-e/2,\chi_D)}{L(s-e/2,\chi_D)} \prod_{\mathrm{bad}\, p} (1 - p^{e-s}) L_p(s-1): & n = 0; \end{cases}$$ if $e$ is even and $$\tilde L(n,\gamma,s) = \begin{cases} \frac{L(s-(e+1)/2, \chi_{\mathcal{D}})}{\zeta(2s-1-e)} \prod_{\mathrm{bad}\, p} \frac{1 - p^{e-s}}{1 - p^{e+1-2s}} L_p(n,\gamma,s-1): & n \ne 0; \\ \\ \frac{\zeta(2s-2-e)}{\zeta(2s-1-e)} \prod_{\mathrm{bad}\, p} \frac{(1 - p^{e-s})(1 - p^{e+2-2s})}{1 - p^{e+1-2s}} L_p(s-1): & n = 0; \end{cases}$$ if $e$ is odd.
\end{rem}

\begin{rem} We denote by $E_k$ the series $$E_k(\tau) = \mathfrak{e}_0 + \sum_{\gamma \in \Lambda'/\Lambda} \sum_{n > 0} c(n,\gamma,0,y) q^n \mathfrak{e}_{\gamma}.$$ The formula (2) gives $I(k,y,n,0) = (-2\pi i)^k n^{k-1} \frac{1}{\Gamma(k)}$ independently of $y$, and so $E_k(\tau)$ is holomorphic. When $k > 2$, this is just the zero-value $E_k(\tau) = E_k^*(\tau,0)$ and therefore $E_k$ is a modular form. In small weights this tends to fail because the terms $$\lim_{s \rightarrow 0} \tilde L(n,\gamma,k+e/2+2s) I(k,y,n,s)$$ may have a pole of $\tilde L$ cancelling the zero of $I$ for $n \le 0$, resulting in nonzero (and often nonholomorphic) contributions to $E_k^*(\tau,0).$
\end{rem}

\begin{rem} Suppose the dimension $e$ is even; then we can apply theorem 4.8 of \cite{BK} to get a simpler coefficient formula. (The condition $k = e/2$ there is only necessary for their computation of local $L$-factors, which we do not use.) It follows that the coefficient $c(n,0)$ of $q^n \mathfrak{e}_0$ in $E_k$ is $$c(n,0) = \frac{(2\pi)^k}{L(k,\chi_D) \sqrt{|\Lambda'/\Lambda|} \Gamma(k)} \cdot \sigma_{k-1}(n,\chi_D) \cdot \prod_{p | D'}\Big[ (1 - p^{e/2 - k}) L_p(n,0,k + e/2 - 1)\Big],$$ where $\sigma_{k-1}(n,\chi_D)$ is the twisted divisor sum $$\sigma_{k-1}(n,\chi_D) = \sum_{d | n} \chi_D(n/d) d^{k-1}$$ and $D' = 4 |\Lambda'/\Lambda|.$ For a fixed lattice $\Lambda$, the expression $\prod_{p | D'} \Big[ (1 - p^{e/2 - k}) L_p(n,0,k+e/2 - 1) \Big]$ can always be worked out exactly using the method of \cite{CKW}, although this can be somewhat tedious (in particular the case $p=2$, which was worked out explicitly in the appendix of \cite{W}.) A worksheet in SAGE to compute these expressions is available on the author's university webpage, and was used to compute the examples in the following sections. Theorem 4.8 of \cite{BK} also gives an interpretation of the coefficients when $e$ is odd but this is more complicated.
\end{rem}


\section{Weight $1$}

In weight $1$, the $L$-series term is always holomorphic at $s = 0$. However, the zero-value $$I(1,y,0,0) = -i\pi$$ being nonzero means that $E_k$ still needs a correction term. Setting $s = 0$ in the real-analytic Eisenstein series gives \begin{align*} E_1^*(\tau,0) &= E_1(\tau) - \pi \frac{(-1)^{(2 + b^- - b^+) / 4}}{\sqrt{|\Lambda'/\Lambda|}} \frac{L(0,\chi_D)}{L(1,\chi_D)} \times \\ &\quad\quad\quad\quad \times\sum_{\substack{\gamma \in \Lambda'/\Lambda \\ q(\gamma) \in \mathbb{Z}}} \Big[ \prod_{\mathrm{bad}\, p}   \lim_{s \rightarrow 0} (1 - p^{e/2 - 1 - 2s}) L_p(0,\gamma,e/2 + 2s) \Big] \mathfrak{e}_{\gamma},\end{align*} where $D$ is the discriminant $D = -4|\Lambda'/\Lambda|$ and the bad primes are the primes dividing $D$. In particular, $E_1$ may differ from the true modular form $E_1^*(\tau,0)$ by a constant. (Of course, $E_1^*(\tau,0)$ may be identically zero.) \\

For two-dimensional negative-definite lattices, the corrected Eisenstein series $E_1^*(\tau,0)$ is often a multiple of the theta series. This leads to identities relating representation numbers of quadratic forms and divisor counts. Of course, such identities are well-known from the theory of modular forms of higher level. The vector-valued proofs tend to be shorter since $M_k(\rho^*)$ is generally much smaller than the space of modular forms of higher-level in which the individual components lie, so there is less algebra (although computing the local factors takes some work). We give two examples here.\\

\begin{ex} Consider the quadratic form $q(x,y) = -x^2 - xy - y^2$, with $|\Lambda'/\Lambda| = 3.$ The $L$-function values are $$L(0,\chi_{-12}) = \frac{2}{3}, \; \; L(1,\chi_{-12}) = \frac{\pi \sqrt{3}}{6}$$ and the local $L$-series are $$L_2(0,0,s) = \frac{1 + 2^{-s}}{1 - 2^{2 - 2s}}, \; \; L_3(0,0,s) = \frac{1}{1 - 3^{1-s}}$$ with $$\lim_{s \rightarrow 0} (1 - 2^{-2s}) L_2(0,0,1 + 2s) = \frac{3}{4}, \; \; \lim_{s \rightarrow 0} (1 - 3^{-2s}) L_3(0,0,1+2s) = 1,$$ and therefore $E_1^*(\tau,0) = E_1(\tau) + \mathfrak{e}_0.$ Since $M_1(\rho^*)$ is one-dimensional, comparing constant terms shows that $$E_1(\tau) + \mathfrak{e}_0 = 2\vartheta.$$ Using remark 11, we find that the coefficient $c(n,0)$ of $q^n \mathfrak{e}_0$ in $E_1$ is \begin{align*} c(n,0) &= \frac{2\pi}{L(1,\chi_{-12}) \cdot \sqrt{3}} \cdot \sigma_0(n,\chi_{-12}) \times \\ &\quad \quad \quad \quad \times \underbrace{\begin{cases} 3/2: & v_2(n)\, \mathrm{even}; \\ 0: & v_2(n) \, \mathrm{odd}; \end{cases}}_{\mathrm{local}\, \mathrm{factor}\, \mathrm{at}\, 2} \cdot \underbrace{\begin{cases} 2: & n \ne (3a+2) 3^b \; \mathrm{for}\; \mathrm{any} \; a,b \in \mathbb{N}_0; \\ 0: & n = (3a + 2)3^b \; \mathrm{for} \; \mathrm{some} \; a,b \in \mathbb{N}_0; \end{cases}}_{\mathrm{local}\, \mathrm{factor} \, \mathrm{at}\, 3} \\ \\ &= 12 \Big[\sum_{d | n} \Big( \frac{-12}{d} \Big) \Big] \cdot \begin{cases} 1: & n \ne (3a+2)3^b; \\ 0: & n = (3a+2)3^b. \end{cases}\end{align*} This implies the identity \begin{align*} &\quad \#\{(a,b) \in \mathbb{Z}^2: \, a^2 + ab + b^2 = n \} \\ &= 6 \varepsilon \cdot \Big( \#\{\mathrm{divisors} \, d = 6 \ell + 1 \, \mathrm{of} \, n\} - \#\{\mathrm{divisors} \, d = 6 \ell - 1 \, \mathrm{of} \, n\} \Big), \end{align*} valid for $n \ge 1$, where $\varepsilon = 1$ unless $n$ has the form $(3a+2)3^b$ for $a,b \in \mathbb{N}_0$, in which case $\varepsilon = 0.$
\end{ex}

\begin{ex} Consider the quadratic form $q(x,y) = -x^2 - y^2$, with $|\Lambda'/\Lambda| = 4$ and $\chi_{-16} = \chi_{-4}.$ The $L$-function values are $$L(0,\chi_{-4}) = \frac{1}{2}, \; \; L(1,\chi_{-4}) = \frac{\pi}{4},$$ and the only bad prime is $2$ with $L_2(0,0,s) = \frac{1}{1 - 2^{1-s}}$ and therefore $$\lim_{s \rightarrow 0} (1 - 2^{e/2 - 1 - 2s}) L_2(0,0,e/2 + 2s) = 1.$$ Therefore, $$E_1^*(\tau,0) = E_1(\tau) + \mathfrak{e}_0.$$ Since $M_1(\rho^*)$ is one-dimensional, comparing constant terms gives $E_1(\tau) + \mathfrak{e}_0 = 2 \vartheta(\tau).$

By remark 11, the coefficient $c(n,0)$ of $q^n \mathfrak{e}_0$ in $E_1$ is $$c(n,0) = \frac{2\pi}{L(1,\chi_{-4}) \cdot 2} \cdot \sigma_0(n,\chi_{-4}) \cdot \underbrace{\begin{cases}2 : & \Big( \frac{-4}{n} \Big) \ne -1; \\ 0: & \Big( \frac{-4}{n} \Big) = -1; \end{cases}}_{\mathrm{local}\, \mathrm{factor} \, \mathrm{at}\, 2} = 8 \sum_{d | n} \Big( \frac{-4}{d} \Big),$$ and therefore \begin{align*} &\quad \#\Big\{ (a,b) \in \mathbb{Z}^2: \, a^2 + b^2 = n \Big\} = 4 \sum_{d | n} \Big( \frac{-4}{d} \Big) \\ &= 4 \cdot \Big( \#\{\mathrm{divisors} \, d = 4\ell + 1 \, \mathrm{of} \, n\} - \#\{\mathrm{divisors}\, d = 4\ell + 3 \, \mathrm{of} \, n\} \Big). \end{align*}
\end{ex}



\begin{rem} Experimentally one often finds that the weight-$1$ Eisenstein series attached to a discriminant form equals a theta series even in cases where it is impossible to associate a weight $1$ theta series to the discriminant form in a meaningful sense; such relations are almost certainly coincidence resulting from small cusp spaces in weight $1$. For example, the indefinite lattice with Gram matrix $$S = \begin{pmatrix} 2 & -1 & -1 & -1 \\ -1 & 2 & -1 & -1 \\ -1 & -1 & 2 & -1 \\ -1 & -1 & -1 & 2 \end{pmatrix}$$ yields an Eisenstein series in which the component of $\mathfrak{e}_0$ is $$E_1^*(\tau,0) = \frac{2}{3} + 4q  + 4q^3 + 4q^4 + 8q^7 + 4q^9 + ...$$ i.e. $\frac{2}{3}$ times the theta series of the quadratic form $x^2 + xy + y^2$. However, the discriminant form of $S$ has signature $2 \, \bmod \, 8$ and is therefore not represented by a negative-definite lattice whose theta series has weight one. \\

On the other hand, replacing $S$ by $$-3S = \begin{pmatrix} -6 & 3 & 3 & 3 \\ 3 & -6 & 3 & 3 \\ 3 & 3 & -6 & 3 \\ 3 & 3 & 3 & -6 \end{pmatrix}$$ yields an Eisenstein series in which the component of $\mathfrak{e}_0$ is $$E_1^*(\tau,0) = \frac{34}{27} - \frac{4}{9} q + \frac{68}{9} q^3 - \frac{4}{9} q^4 - \frac{8}{9} q^7 + \frac{68}{9} q^9 \pm ...$$ with the surprising property that its coefficients have infinitely many sign changes; in particular, this example should make clear that $E_1^*(\tau,0)$ is not simply a theta series for every lattice.
\end{rem}

\section{Weight $3/2$}

In weight $3/2$, the $L$-series term is \begin{align*} &\quad \tilde L(n,\gamma,3/2 + e/2 + 2s) \\ &= \begin{cases} \frac{L(1+2s,\chi_{\mathcal{D}})}{\zeta(4s+2)} \prod_{\mathrm{bad}\, p} \frac{ 1 - p^{(e-3)/2 - 2s}}{1 - p^{-2 - 4s}} L_p(n,\gamma,1/2 + e/2 + 2s): & n \ne 0; \\ \\ \frac{\zeta(4s+1)}{\zeta(4s+2)} \prod_{\mathrm{bad}\, p} \frac{(1 - p^{(e-3)/2 - 2s})(1 - p^{-1-4s})}{1 - p^{-2-4s}} L_p(n,\gamma,1/2 + e/2 + 2s): & n = 0; \end{cases}\end{align*} and it is holomorphic in $s=0$ unless $n = 0$ or $$\mathcal{D} = -2n d_{\gamma}^2 |\Lambda'/\Lambda| \prod_{\mathrm{bad}\, p} p^2$$ is a square. In these cases, $\tilde L(n,\gamma,3/2 + e/2 + 2s)$ has a simple pole with residue $$\frac{3}{\pi^2} \prod_{\mathrm{bad}\, p}\lim_{s \rightarrow 0} \frac{ (1 - p^{e/2 - 3/2-2s})(1 - p^{-1})}{1 - p^{-2}} L_p(n,\gamma,1/2 + e/2+2s)$$ if $n \ne 0$, and $$\frac{3}{2\pi^2} \prod_{\mathrm{bad}\, p} \lim_{s \rightarrow 0} \frac{(1 - p^{e/2 - 3/2 - 2s})(1 - p^{-1})}{1 - p^{-2}} L_p(n,\gamma,1/2 + e/2 + 2s)$$ if $n = 0.$ \\

This pole cancels with the zero of $I(k,y,n,s)$ at $s=0$, whose derivative there is $$\frac{d}{ds} \Big|_{s=0} I(k,y,n,s) = -16\pi^2 (1+i) y^{-1/2} \beta(4\pi |n|y), \; \; \mathrm{where} \; \beta(x) = \frac{1}{16\pi} \int_1^{\infty} u^{-3/2} e^{-xu} \, \mathrm{d}u,$$ as calculated in \cite{HZ}, section 2.2. This expression is also valid for $n =0$, where it reduces to $$\frac{d}{ds}\Big|_{s=0} I(k,y,0,s) = 2\pi (-i)^{3/2} \frac{d}{ds}\Big|_{s=0} 2^{-1/2-2s} y^{-1/2 - s} \frac{\Gamma(2s+1/2)}{\Gamma(s) \Gamma(s+3/2)} = -\frac{2\pi}{\sqrt{y}} (1+i).$$

Therefore, $E_{3/2}^*(\tau,0)$ is a harmonic weak Maass form that is not generally holomorphic:

\begin{align*}&\quad E_{3/2}^*(\tau,0) \\ &= E_{3/2}(\tau) + \frac{3 (-1)^{(3 + b^+ - b^-)/4} \sqrt{2}}{\pi \sqrt{y |\Lambda'/\Lambda|}} \Big( \sum_{\substack{\gamma \in \Lambda'/\Lambda \\ q(\gamma) \in \mathbb{Z}}} \prod_{p | \#(\Lambda'/\Lambda)} \frac{1 - p^{(e-3)/2}}{1 + p^{-1}} L_p(0,\gamma,1/2 + e/2) \mathfrak{e}_{\gamma}\Big) +  \\ &\quad + \frac{48 (-1)^{(3 + b^+ - b^-)/4} \sqrt{2}}{\sqrt{y |\Lambda'/\Lambda|}}  \sum_{\substack{\gamma \in \Lambda'/\Lambda \\ n \in \mathbb{Z} - q(\gamma) \\ -2n |\Lambda'/\Lambda| = \square}} \Big[ \beta(4\pi |n|y) \prod_{\mathrm{bad}\, p} \frac{1 - p^{(e-3)/2}}{1+ p^{-1}}L_p(n,\gamma,1/2 + e/2) \Big] q^n \mathfrak{e}_{\gamma} ,\end{align*} where $-2n |\Lambda'/\Lambda| = \square$ means that $-2n |\Lambda'/\Lambda|$ should be a rational square. (In particular, the real-analytic correction involves only exponents $n \le 0$.)

\begin{ex} Zagier's Eisenstein series \cite{HZ} occurs as the Eisenstein series for the quadratic form $q(x) = x^2$. The underlying harmonic weak Maass form is \begin{align*} E_{3/2}^*(\tau,0) &= E_{3/2}(\tau) - \frac{3}{\pi \sqrt{y}}\mathfrak{e}_0 - \frac{48}{\sqrt{y}} \sum_{\gamma \in \Lambda'/\Lambda} \sum_{\substack{n \in \mathbb{Z}-q(\gamma) \\ -n = \square}} \beta(4\pi|n| y) \underbrace{\prod_{\mathrm{bad}\, p} \frac{1 - p^{-1}}{1 + p^{-1}} L_p(n,\gamma,1)}_{=1} q^n \mathfrak{e}_{\gamma} \\ &= E_{3/2}(\tau) - \frac{24}{\sqrt{y}} \sum_{n=-\infty}^{\infty} \beta(4\pi (n/2)^2 y) q^{-(n/2)^2} \mathfrak{e}_{n/2}. \end{align*} The coefficient of $q^{n/4}$ in $$E_{3/2}(\tau) = \Big( 1 - 6q - 12q^2 - 16q^3 - ... \Big) \mathfrak{e}_0 + \Big( -4q^{3/4} - 12q^{7/4} - 12q^{11/4} - ... \Big) \mathfrak{e}_{1/2}$$ is $-12$ times the Hurwitz class number $H(n).$ (We obtain Zagier's Eisenstein series in its usual form by summing the components, replacing $\tau$ by $4\tau$ and $y$ by $4y$, and dividing by $-12$.)
\end{ex}

\begin{rem} We can use essentially the same argument as Hirzebruch and Zagier \cite{HZ} to derive the transformation law of the general $E_{3/2}.$ Write $E_{3/2}^*(\tau,0)$ in the form $$E_{3/2}^*(\tau,0) = E_{3/2} + \frac{1}{\sqrt{y}} \sum_{\gamma \in \Lambda'/\Lambda} \sum_{\substack{n \in \mathbb{Z}-q(\gamma) \\ n \le 0}} a(n,\gamma) \beta(-4\pi n y) q^n \mathfrak{e}_{\gamma}$$ with coefficients $a(n,\gamma).$ Applying the $\xi$-operator $\xi = y^{3/2} \overline{\frac{\partial}{\partial \overline{\tau}}}$ of \cite{BF} to $E_{3/2}^*(\tau,0)$ and using $$\frac{d}{dy} \Big[ \frac{1}{\sqrt{y}} \beta(y)\Big] = \frac{1}{16\pi} \frac{d}{dy} \Big[ \int_y^{\infty} v^{-3/2} e^{-v}\, \mathrm{d}v \Big] = -\frac{1}{16\pi} y^{-3/2} e^{-y}$$ shows that the ``shadow" $$\vartheta(\tau) = \sum_{\gamma,n} a(n,\gamma) q^{-n} \mathfrak{e}_{\gamma}$$ is a modular form of weight $1/2$ for the representation $\rho$ (not its dual!), and \begin{align*} E_{3/2}^*(\tau,0) - E_{3/2}(\tau) &= y^{-1/2} \sum_{\gamma \in \Lambda'/\Lambda} \sum_{n \in \mathbb{Z} - q(\gamma)} a(n,\gamma) \beta(-4\pi n y) q^n \mathfrak{e}_{\gamma} \\ &= \frac{1}{16\pi} y^{-1/2} \int_1^{\infty} \sum_{\gamma,n} u^{-3/2} e^{-4\pi n u y} q^n \mathfrak{e}_{\gamma} \, \mathrm{d}u \\ &= \frac{1}{16\pi} y^{-1/2} \int_1^{\infty} u^{-3/2} \vartheta(2iuy - \tau) \, \mathrm{d}u \\ &= \frac{\sqrt{2i}}{16\pi} \int_{-x+iy}^{i \infty} (v + \tau)^{-3/2} \vartheta(v) \, \mathrm{d}v, \; \; v = 2iuy - \tau.\end{align*}

For any $M = \begin{pmatrix} a & b \\ c & d \end{pmatrix} \in Mp_2(\mathbb{Z}),$ defining $\tilde M = \begin{pmatrix} a & -b \\ -c & d \end{pmatrix}$ as in remark 6 and substituting $v = \tilde M \cdot t$ gives \begin{align*} E_{3/2}^*(M \cdot \tau,0) - E_{3/2}(M \cdot \tau) &= \frac{\sqrt{2i}}{16\pi} \int_{-\overline{M \cdot \tau}}^{i \infty} \Big( \frac{a \tau + b}{c \tau + d} + v \Big)^{-3/2} \vartheta(v) \, \mathrm{d}v \\ &= \frac{\sqrt{2i}}{16\pi} \int_{-\overline{\tau}}^{d/c} \Big( \frac{a \tau + b}{c \tau + d} + \frac{a t - b}{-ct + d} \Big)^{-3/2} \vartheta(\tilde M \cdot t) \, \frac{\mathrm{d}t}{(ct - d)^2} \\ &= \frac{\sqrt{2i}}{16\pi} (c \tau + d)^{3/2} \int_{-\overline{\tau}}^{d/c} (\tau + t)^{-3/2} \rho(\tilde M) \vartheta(t) \, \mathrm{d}t \\ &= (c \tau + d)^{3/2} \rho^*(M) \Big[ \frac{\sqrt{2i}}{16\pi} \int_{-\overline{\tau}}^{d/c} (\tau + t)^{-3/2} \vartheta(t) \, \mathrm{d}t \Big]. \end{align*} Since $E_{3/2}^*(M \cdot \tau,0) = (c \tau + d)^{3/2} \rho^*(M) E_{3/2}^*(\tau,0)$, we conclude that \begin{equation} E_{3/2}(M \cdot \tau) = (c \tau + d)^{3/2} \rho^*(M) \Big[ E_{3/2}(\tau) + \frac{\sqrt{2i}}{16 \pi} \int_{d/c}^{i \infty} (\tau + t)^{-3/2} \vartheta(t) \, \mathrm{d}t \Big]. \end{equation}
\end{rem}

\begin{rem} The transformation law (6) can be used to give an easier sufficient condition for when $E_{3/2}$ is actually a modular form. For example, one can show that $M_{1/2}(\rho) = 0$ for the quadratic form $q(x,y,z) = -x^2 - y^2 - z^2$, which implies that the series $\vartheta$ defined above must be identically $0$ and therefore $$E_{3/2}(M \cdot \tau) = (c \tau + d)^{3/2} \rho^*(M) E_{3/2}(\tau),$$ so $E_{3/2}$ is a true modular form. (In this case, the local $L$-series $L_2(n,\gamma,2+2s)$ at $p=2$ is holomorphic at $s=0$, and therefore the factor $(1 - 2^{-2s})$ annihilates the $L$-series term $\tilde L(n,\gamma,3/2 + e/2)$ in the shadow.) This must be the theta series because $M_{3/2}(\rho^*)$ is one-dimensional. \\

It may be worth pointing out that the coefficient formulas (\cite{BK}, theorem 4.8) for this theta series and for the Zagier Eisenstein series are nearly identical, since the squarefree parts of their discriminant and the ``bad primes'' are the same: the only real difference between them is the local factor at $2$. For odd integers $n$, the local factor at $2$ is easily computed and in both cases depends only on the remainder of $n$ mod $8$, so the coefficients $r_3(n)$ of the theta series and $H(4n)$ of the Zagier Eisenstein series within these congruence classes are proportional. Specifically, $$r_3(n) = 12 H(4n), \; n \equiv 1,5 \, (8); \; \; r_3(n) = 6 H(4n), \; n \equiv 3\, (8); \; \; r_3(n) = 0, \; n \equiv 7 \, (8).$$ These identities are well-known and were already proved by Gauss.
\end{rem}

\begin{ex} Even when $M_{1/2}(\rho) \ne 0$, we can identify $\vartheta$ in $M_{1/2}(\rho)$ by computing finitely many coefficients. Let $q$ be the quadratic form $q(x,y,z) = x^2 + y^2 - z^2.$ The space $M_{1/2}(\rho)$ is always spanned by unary theta series embedded into $\mathbb{C}[\Lambda'/\Lambda]$ (as proven by Skoruppa \cite{SK}) and in this case one can find the basis \begin{align*} & \vartheta_1(\tau) = \Big( 1 + 2q + 2q^4 + ... \Big) (\mathfrak{e}_{(0,0,0)} + \mathfrak{e}_{(1/2,0,1/2)}) +\\ &\quad\quad\quad\quad + \Big( 2q^{1/4} + 2q^{9/4} + 2q^{25/4} + ... \Big) (\mathfrak{e}_{(0,1/2,0)} + \mathfrak{e}_{(1/2,1/2,1/2)}), \\ & \vartheta_2(\tau) = \Big( 1 + 2q + 2q^4 + ... \Big) (\mathfrak{e}_{(0,0,0)} + \mathfrak{e}_{(0,1/2,1/2)}) +\\ &\quad\quad\quad\quad + \Big( 2q^{1/4} + 2q^{9/4} + 2q^{25/4} + ... \Big) (\mathfrak{e}_{(1/2,0,0)} + \mathfrak{e}_{(1/2,1/2,1/2)}). \end{align*} The local $L$-series at the bad prime $p=2$ for the constant term $n=0$ are $$(1 - 2^{-2s}) L_p(0,0,2+2s) = \frac{1}{1 - 2^{-1-4s}} \; \; \mathrm{and} \; \;  (1 - 2^{-2s}) L_p(0,\gamma,2+2s) = 1$$ for $\gamma \in \{(1/2,0,1/2),(0,1/2,1/2)\},$ which implies that $$E_{3/2}^*(\tau,0) = E_{3/2}(\tau) - \frac{3}{2\pi \sqrt{y}} \Big( \frac{4}{3} \mathfrak{e}_{(0,0,0)} + \frac{2}{3} \mathfrak{e}_{(1/2,0,1/2)} + \frac{2}{3} \mathfrak{e}_{(0,1/2,1/2)} \Big) + \Big( \mathrm{negative} \, \mathrm{powers} \, \mathrm{of} \, q\Big)$$ and therefore the shadow in equation (6) must be $$\vartheta(\tau) = -8 \Big( \vartheta_1(\tau) + \vartheta_2(\tau) \Big).$$ In particular, the $\mathfrak{e}_0$-component $E_{3/2}(\tau)_0$ of $E_{3/2}(\tau)$ is a mock modular form of level $4$ that transforms under $\Gamma(4)$ by $$E_{3/2}(M \cdot \tau)_0 = (c \tau + d)^{3/2} \Big[ E_{3/2}(\tau)_0 - \frac{\sqrt{2i}}{\pi} \int_{d/c}^{i \infty} ( \tau + t)^{-3/2} \Theta(t) \, \mathrm{d}t \Big], \; \; \Theta(t) = \sum_{n \in \mathbb{Z}} \mathbf{e}(n^2 t).$$ It was shown by Bringmann and Lovejoy \cite{BL} that the series $$\overline{\mathcal{M}}(\tau + 1/2) = 1 - \sum_{n=1}^{\infty} | \overline{\alpha}(n)| q^n = 1 - 2q - 4q^2 - 8q^3 - 10q^4 - ...$$ of example $4$, where $|\overline{\alpha}(n)|$ counts overpartition rank differences of $n$, has the same transformation behavior under the group $\Gamma_0(16)$, which implies that the difference between $\overline{\mathcal{M}}(\tau + 1/2)$ and the $\mathfrak{e}_0$-component of $E_{3/2}$ is a true modular form of level $16$. We can verify that these are the same by comparing all Fourier coefficients up to the Sturm bound.
\end{ex}

\section{Weight $2$}

In weight $k = 2,$ the $L$-series term is $$\tilde L(n,\gamma,2 + e/2 + 2s) = \begin{cases} \frac{1}{L(1+2s,\chi_D)} \prod_{\mathrm{bad}\, p} (1 - p^{e/2 - 2 - 2s}) L_p(n,\gamma,1 + e/2 + 2s): & n \ne 0; \\ \\ \frac{L(1+2s,\chi_D)}{L(2+2s,\chi_D)} \prod_{\mathrm{bad}\, p} (1 - p^{e/2 - 2 - 2s}) L_p(n,\gamma,1 + e/2 + 2s): & n = 0. \end{cases}$$

Since $L(1,\chi)$ is never zero for any Dirichlet character, the only way a pole can occur at $s=0$ is if $n = 0$ and $D = |\Lambda'/\Lambda|$ is square. (In particular, when $|\Lambda'/\Lambda|$ is not square, $E_2$ is a modular form.) \\

Assume that $|\Lambda'/\Lambda|$ is square. Then $$L(1+2s,\chi_D) = \zeta(1+2s) \prod_{\mathrm{bad}\, p} (1 - p^{-1-2s}),$$ and therefore $\tilde L(0,\gamma,2 + e/2 + 2s)$, has a simple pole at $s=0$ with residue \begin{align*} &\quad \mathrm{Res}\Big( \tilde L(0,\gamma,2 + e/2 + 2s), s=0 \Big) \\ &= \frac{1}{2 L(2,\chi_D)} \prod_{\mathrm{bad}\, p}\Big[  (1 - p^{-1}) \lim_{s \rightarrow 0} (1 - p^{e/2 - 2 - 2s}) L_p(0,\gamma,1 + e/2 + 2s) \Big] \\ &= \frac{3}{\pi^2} \lim_{s \rightarrow 0} \prod_{\mathrm{bad}\, p} \frac{1 - p^{e/2 - 2 - 2s}}{1 + p^{-1}} L_p(0,\gamma,1 + e/2 + 2s) \end{align*} for any $\gamma \in \Lambda'/\Lambda$ with $q(\gamma) \in \mathbb{Z}.$ This pole is canceled by the zero of $I(2,y,0,s)$ at $s=0$ which has derivative \begin{align*} \frac{d}{ds} \Big|_{s=0} I(2,y,0,s) &= -2\pi (2y)^{-1} \frac{d}{ds}\Big|_{s=0} (2y)^{-2s} \frac{\Gamma(2s+1)}{\Gamma(s)\Gamma(s+2)} \\ &= -\frac{\pi}{y}, \end{align*} so \begin{equation} E_2^*(\tau,0) = E_2(\tau) - \frac{3}{\pi y \sqrt{|\Lambda'/\Lambda|}} \lim_{s \rightarrow 0} \sum_{\substack{\gamma \in \Lambda'/\Lambda \\ q(\gamma) \in \mathbb{Z}}} \prod_{\mathrm{bad}\, p} \frac{1 - p^{e/2 - 2 - 2s}}{1 + p^{-1}} L_p(0,\gamma,1 + e/2 + 2s) \mathfrak{e}_{\gamma}. \end{equation}

\begin{ex} Let $\Lambda$ be a unimodular lattice. The only bad prime is $p=2$. Using the hyperbolic plane $q(x,y) = xy$ to define $\Lambda$, the local $L$-function is $$L_2(0,0,s) = \frac{1 - 2^{-s}}{(1 - 2^{1-s})^2}$$ with $L_2(0,0,2) = 3$, so we obtain the well-known result $$E_2^*(\tau,0) = E_2(\tau) - \frac{3}{\pi y} \cdot \frac{1 - 1/2}{1 + 1/2} L_2(0,0,2) = E_2(\tau) - \frac{3}{\pi y}.$$
\end{ex}

\begin{rem} We can summarize the above by saying that $$E_2^*(\tau,0) = E_2(\tau) - \frac{1}{y} \sum_{\substack{\gamma \in \Lambda'/\Lambda \\ q(\gamma) \in \mathbb{Z}}} A(\gamma) \mathfrak{e}_{\gamma}$$ is a Maass form for some constants $A(\gamma)$. For $M = \begin{pmatrix} a & b \\ c & d \end{pmatrix} \in Mp_2(\mathbb{Z}),$ since $$E_2^*(M \cdot \tau,0) = (c \tau + d)^2 \rho^*(M) E_2^*(\tau,0),$$ we find the transformation law \begin{align*} E_2(M \cdot \tau) &= E_2^*(M \cdot \tau,0) + \frac{|c \tau + d|^2}{y} \sum_{\substack{\gamma \in \Lambda'/\Lambda \\ q(\gamma) \in \mathbb{Z}}} A(\gamma) \mathfrak{e}_{\gamma} \\ &= (c \tau + d)^2 \Big[ \rho^*(M) E_2(\tau) - 2ic (c\tau + d) \sum_{q(\gamma) \in \mathbb{Z}} A(\gamma) \rho^*(M) \mathfrak{e}_{\gamma} \Big]. \end{align*}
\end{rem}

\begin{ex} The weight-$2$ Eisenstein series for the quadratic form $q(x,y) = x^2 + 3xy + y^2$ is a true modular form because the discriminant $5$ of $q$ is not a square. In particular, the $\mathfrak{e}_0$-component $$1 -30q - 20q^2 - 40q^3 - 90q^4 - 130q^5 - 60q^6 - 120q^7 - 100q^8 - 210q^9 - ...$$  is a modular form of weight $2$ for the congruence subgroup $\Gamma_1(5).$ Using remark 11, we see that the coefficient $c(n)$ of $q^n$ for $n$ coprime to $10$ is $$c(n) = \begin{cases} -30 \sum_{d | n} \Big( \frac{5}{n/d} \Big) d: & n \equiv \pm 1 \, \bmod\, 10; \\ -20 \sum_{d | n} \Big( \frac{5}{n/d} \Big) d : & n \equiv \pm 3 \, \bmod \, 10; \end{cases}$$ with a more complicated expression for other $n$ involving the local factors at $2$ and $5$.
\end{ex}

\begin{ex} The weight-$2$ Eisenstein series for the quadratic form $q(x,y) = 2xy$ is \begin{align*} E_2(\tau) &= \Big( 1 -8q - 40q^2 - 32q^3 - 104q^4 - ... \Big) \mathfrak{e}_{(0,0)} \\ &\quad\quad + \Big( -16q -32q^2 - 64q^3 - 64q^4 - 96q^5 - ... \Big) (\mathfrak{e}_{(0,1/2)} + \mathfrak{e}_{(1/2,0)}) \\ &\quad\quad + \Big( -8q^{1/2} - 32q^{3/2} - 48q^{5/2} - 64q^{7/2} - 104q^{9/2} - ... \Big) \mathfrak{e}_{(1/2,1/2)} \\ &= \Big( 1 - 8 \sum_{n=1}^{\infty} \Big[ \sum_{d | 2n} (-1)^d d \Big] q^n \Big) \mathfrak{e}_{(0,0)} \\ &\quad\quad + \Big( -8 \sum_{n=1}^{\infty} \Big[ \sum_{d | n} (1 - (-1)^{n/d}) d \Big] q^n \Big) (\mathfrak{e}_{(0,1/2)} + \mathfrak{e}_{(1/2,0)}) \\ &\quad\quad + \Big( -8 \sum_{n=0}^{\infty} \sigma_1(2n+1) q^{n+1/2} \Big) \mathfrak{e}_{(1/2,1/2)}. \end{align*} It is not a modular form. (In fact, there are no modular forms of weight $2$ for this lattice.) On the other hand, the real-analytic correction (7) only involves the components $\mathfrak{e}_{\gamma}$ for which $q(\gamma) \in \mathbb{Z}$, i.e. $\mathfrak{e}_{(0,0)}, \mathfrak{e}_{(0,1/2)}, \mathfrak{e}_{(1/2,0)}$, so the components $$1 - 8 \sum_{n=1}^{\infty} \Big[ \sum_{d | 2n} (-1)^d d \Big] q^n, \; \; \sum_{n=1}^{\infty} \Big[ \sum_{d|n} (1 - (-1)^{n/d}) d \Big] q^n$$ are only quasimodular forms of level $4$, while $\sum_{n=0}^{\infty} \sigma_1(2n+1) q^{2n+1}$ is a true modular form.
\end{ex}

\begin{ex} Although the discriminant group of the quadratic form $q(x_1,x_2,x_3,x_4) = -x_1^2 - x_2^2 - x_3^2 - x_4^2$ has square order $16$, the correction term still vanishes in this case. This is because the local $L$-functions for $p=2$, $$L_2(0,\gamma, 3+s) = \begin{cases} \frac{2 + 2^{-s}}{2 - 2^{-s}}: & \gamma = 0; \\ 1: & \gamma = (1/2,1/2,1/2,1/2); \end{cases}$$ are both holomorphic at $s=0$ and therefore annihilated by the term $(1 - p^{4/2 - 2 - 2s})$ at $s=0$. (Another way to see this is that $\sum_{\substack{\gamma \in \Lambda'/\Lambda \\ q(\gamma) \in \mathbb{Z}}} A(\gamma) \mathfrak{e}_{\gamma}$ is invariant under $\rho$ due to the transformation law of $E_2$, but there are no nonzero invariants of $\rho$ in this case.) In fact, the Eisenstein series $E_2$ for this lattice is exactly the theta series as one can see by calculating the first few coefficients. Comparing coefficients of the $\mathfrak{e}_0$-component leads immediately to Jacobi's formula: \begin{align*} \#\{(a,b,c,d) \in \mathbb{Z}^4: \, a^2 + b^2 + c^2 + d^2 = n\} &= \frac{(2\pi)^2}{L(2,\chi_{64}) \cdot 4} \cdot \sigma_1(n,\chi_{64}) \cdot L_2(n,0,3) \\ &= 8 \cdot \Big[ \sum_{d | n} \Big( \frac{4}{n/d} \Big) d \Big] \cdot \begin{cases} 1: & n \, \mathrm{odd}; \\ 3 \cdot 2^{-v_2(n)}: & n \, \mathrm{even}; \end{cases} \\ &= \begin{cases} 8 \sum\limits_{d | n} d: & n \, \mathrm{odd}; \\  \\ 24 \sum\limits_{\substack{d | n \\ d \, \mathrm{odd}}} d : & n \, \mathrm{even}; \end{cases} \end{align*} for all $n \in \mathbb{N}.$
\end{ex}

\section{Weight $1/2$}


The Fourier expansion (1) is no longer valid in weight $k = 1/2$; in fact, the $L$-series factor in this case is $$\tilde L \Big(n,\gamma, \frac{e+1}{2} + 2s \Big) = \begin{cases} \frac{L_{\mathcal{D}}(2s)}{\zeta(4s)} \prod_{\mathrm{bad}\, p} \frac{1 - p^{(e-1)/2 - 2s}}{1 - p^{-4s}} L_p \Big(n,\gamma, \frac{e-1}{2} + 2s \Big): & n \ne 0; \\ \\ \frac{\zeta(4s-1)}{\zeta(4s)} \prod_{\mathrm{bad}\, p} \frac{(1 - p^{(e-1)/2 - 2s})(1 - p^{1-4s})}{1 - p^{-4s}} L_p \Big(n,\gamma, \frac{e-1}{2} + 2s \Big): & n = 0; \end{cases}$$ which generally has a singularity at $s=0$, so our approach fails in weight $1/2$. \\

Despite this, the weight $1/2$ Eisenstein series $E^*_{1/2}(\tau,s)$ should extend analytically to $s=0$. One way to study $E_{1/2}^*(\tau,s)$ is by applying the Bruinier-Funke operator $\xi_{3/2}$ to the weight $3/2$ series $E_{3/2}^*(\tau,s)$ for the dual representation (i.e. the same lattice with negated quadratic form); from $\xi_{3/2} y^s = -\frac{si}{2} y^{s+1/2}$ one obtains $\xi_{3/2} E_{3/2}^*(\tau,s) = -\frac{si}{2} E_{1/2}^*(\tau,s+1/2)$ for all large enough $s$. Carrying over the arguments from the scalar-valued case (e.g. \cite{DS}, section 4.10) should imply that $E_{1/2}^*(\tau,s)$ will satisfy some functional equation relating $E_{1/2}^*(\tau,s+1/2)$ to $E_{1/2}^*(\tau,-s)$ (or more likely some combination of $E_{1/2,\beta}^*(\tau,-s)$ as $\beta$ runs through elements of $\Lambda'/\Lambda$ with $q(\beta) \in \mathbb{Z}$ in general) although in the half-integer case this seems less straightforward. Assuming this, for large enough $\mathrm{Re}[s]$ it follows that $E_{1/2}^*(\tau,-s)$ should be a linear combination of $\xi_{3/2} E_{3/2,\beta}^*(\tau,s)$ with coefficients depending on $s$ but independent of $\tau$; we might even expect this to hold for arbitrary $s$ and therefore conjecture:


\begin{conj} The zero-value $E_{1/2}^*(\tau,0)$ for a discriminant form $(\Lambda'/\Lambda, q)$ is a holomorphic modular form of weight $1/2$; moreover it is a linear combination of the shadows of mock Eisenstein series $E_{3/2,\beta}(\tau)$ for $(\Lambda'/\Lambda, -q)$ that has constant term $1 \cdot \mathfrak{e}_0$.
\end{conj}

Unfortunately, if this is true then from the point of view of this note there is little motivation to consider $E_{1/2}^*(\tau,0)$ further: modular forms of weight $1/2$ are spanned by what are essentially unary theta series and any resulting identities among coefficients will be uninteresting. There may be interest in higher terms of the Taylor expansion of $E_{1/2}^*(\tau,s)$ in the variable $s$ which might be used to generate mock modular forms of weight $1/2$ and higher depth, but this is outside the scope of this note. \\


There is one class of examples where this conjecture can be verified directly. In dimension $e=1$, where the quadratic form is $q(x) = -mx^2$ for some $m \in \mathbb{N}$, we can make sense of the coefficient formula because the terms $1 - p^{1-4s}$ are cancelled by the numerators at $s=0$, and the Fourier series then provides the analytic continuation of $E_{1/2}^*(\tau,s)$ to $s=0$. The $L$-series factor in this case is $$\tilde L(n,\gamma,1 + 2s) = \begin{cases} \frac{L(2s,\chi_{\mathcal{D}})}{\zeta(4s)} \prod_{\mathrm{bad}\, p} \frac{L_p(n,\gamma,2s)}{1 + p^{-2s}}: & n \ne 0; \\ \\ \frac{\zeta(4s-1)}{\zeta(4s)} \prod_{\mathrm{bad}\, p} \frac{(1 - p^{1-4s}) L_p(n,\gamma,2s)}{1 + p^{-2s}}: & n = 0. \end{cases}$$ Here, $\mathcal{D}$ is the discriminant $$\mathcal{D} = 2d_{\gamma}^2 n |\Lambda'/\Lambda| \prod_{\mathrm{bad}\, p} p^2 = 4mn d_{\gamma}^2 \prod_{\mathrm{bad}\, p} p^2.$$ 

Suppose for simplicity that $m$ is squarefree (and in particular, $\beta = 0$ is the only element of $\Lambda'/\Lambda$ with $q(\beta) \in \mathbb{Z}$). The local $L$-factors can be calculated by elementary means (for example, with Hensel's lemma), and the result in this case is that $E_{1/2}(\tau) = 1 \cdot \mathfrak{e}_0 + \sum_{\gamma \in \Lambda'/\Lambda} \sum_{n \in \mathbb{Z} - q(\gamma)} c(n,\gamma) q^n \mathfrak{e}_{\gamma}$ with $$c(n,\gamma) = 2 \cdot (1/2)^{\varepsilon}, \; \; \varepsilon = \#\{\text{primes} \; p \ne 2 \; \text{dividing} \; d_{\gamma}\} + \begin{cases} 1: & 4 | d_{\gamma} \\ 0: & \text{otherwise}. \end{cases}$$ Here, $d_{\gamma}$ is the denominator of $\gamma$; that is, the smallest number for which $d_{\gamma} \gamma \in \Lambda.$ \\

The shadow of the mock Eisenstein series $E_{3/2}(\tau)$ attached to $mx^2$ can be computed directly as well, although this is more difficult. On the other hand, one can use the following trick: via the theta decomposition, the nonholomorphic weight $3/2$ Eisenstein series $E_{3/2}^*(\tau,0)$ corresponds to a nonholomorphic, scalar Jacobi Eisenstein series $E_{2,m}^*(\tau,z,0)$ of index $m$. The argument of chapter $4$ of \cite{EZ} still applies to this situation and in particular $E_{2,m}^*(\tau,z,0) = \frac{1}{\sigma_1(m)} E_{2,1}^*(\tau,z,0) | V_m$ for the Hecke-type operator $$\Phi | V_m(\tau,z) = m \sum_M (c \tau + d)^{-2} \mathbf{e}\Big( -\frac{cmz^2}{c \tau + d} \Big) \Phi \Big( \frac{a \tau + b}{c \tau + d}, \frac{mz}{c \tau + d} \Big),$$ the sum taken over cosets of determinant-$m$ integral matrices $M$ by $SL_2(\mathbb{Z})$. (Here we must assume that $m$ is squarefree). However, $E_{2,1}^*(\tau,z,0)$ arises through the theta decomposition from the Zagier Eisenstein series and so its coefficients are well-known. In this way one can compute $$E_{3/2}^*(\tau,0) = -\frac{12}{\sigma_1(m)} \sum_{\gamma \in \Lambda'/\Lambda} \sum_{n \in \mathbb{Z} - q(\gamma)} \sum_{a | m} a H(4mn / a^2) q^n \mathfrak{e}_{\gamma} + \frac{1}{\sqrt{y}} \sum_{\gamma \in \Lambda'/\Lambda} \sum_{n \in \mathbb{Z} + q(\gamma)} a(n,\gamma) q^{-n} \mathfrak{e}_{\gamma},$$ where $H(n)$ is the Hurwitz class number (and $H(n) = 0$ if $n$ is noninteger) and the coefficients of the shadow are $$a(n,\gamma) = \begin{cases} -24 \sqrt{m} \frac{\sigma_0(m)}{\sigma_1(m)}: & n = 0; \\ -48 \sqrt{m} \frac{\sigma_0(\mathrm{gcd}(m,n))}{\sigma_1(m)}: & mn = \square, \; mn \ne 0; \\ 0: & \text{otherwise}, \end{cases}$$ and where we use the convention $\mathrm{gcd}(m,n) = \prod_{v_p(m),v_p(n) \ge 0} p^{\min(v_p(m),v_p(n))}$ (e.g. $\mathrm{gcd}(30,3/4) = 3$). Unraveling this, we see that $E_{1/2}(\tau)$ differs from the shadow of $E_{3/2}(\tau)$ by the factor $-24 \sqrt{m} \frac{\sigma_0(m)}{\sigma_1(m)}$.

\bibliographystyle{plainnat}
\bibliography{\jobname}

\begin{thebibliography}{15}
\providecommand{\natexlab}[1]{#1}
\providecommand{\url}[1]{\texttt{#1}}
\expandafter\ifx\csname urlstyle\endcsname\relax
  \providecommand{\doi}[1]{doi: #1}\else
  \providecommand{\doi}{doi: \begingroup \urlstyle{rm}\Url}\fi

\bibitem[Bringmann and Lovejoy(2009)]{BL}
Kathrin Bringmann and Jeremy Lovejoy.
\newblock Overpartitions and class numbers of binary quadratic forms.
\newblock \emph{Proc. Natl. Acad. Sci. USA}, 106\penalty0 (14):\penalty0
  5513--5516, 2009.
\newblock ISSN 1091-6490.
\newblock \doi{10.1073/pnas.0900783106}.
\newblock URL \url{http://dx.doi.org/10.1073/pnas.0900783106}.

\bibitem[Bruinier(2002)]{Br}
Jan Bruinier.
\newblock \emph{Borcherds products on {O}(2, {$l$}) and {C}hern classes of
  {H}eegner divisors}, volume 1780 of \emph{Lecture Notes in Mathematics}.
\newblock Springer-Verlag, Berlin, 2002.
\newblock ISBN 3-540-43320-1.
\newblock \doi{10.1007/b83278}.
\newblock URL \url{http://dx.doi.org/10.1007/b83278}.

\bibitem[Bruinier and Funke(2004)]{BF}
Jan Bruinier and Jens Funke.
\newblock On two geometric theta lifts.
\newblock \emph{Duke Math. J.}, 125\penalty0 (1):\penalty0 45--90, 2004.
\newblock ISSN 0012-7094.
\newblock \doi{10.1215/S0012-7094-04-12513-8}.
\newblock URL \url{http://dx.doi.org/10.1215/S0012-7094-04-12513-8}.

\bibitem[Bruinier and K\"uhn(2003)]{BK2}
Jan Bruinier and Ulf K\"uhn.
\newblock Integrals of automorphic {G}reen's functions associated to {H}eegner
  divisors.
\newblock \emph{Int. Math. Res. Not.}, \penalty0 (31):\penalty0 1687--1729,
  2003.
\newblock ISSN 1073-7928.
\newblock \doi{10.1155/S1073792803204165}.
\newblock URL \url{http://dx.doi.org/10.1155/S1073792803204165}.

\bibitem[Bruinier and Kuss(2001)]{BK}
Jan Bruinier and Michael Kuss.
\newblock Eisenstein series attached to lattices and modular forms on
  orthogonal groups.
\newblock \emph{Manuscripta Math.}, 106\penalty0 (4):\penalty0 443--459, 2001.
\newblock ISSN 0025-2611.
\newblock \doi{10.1007/s229-001-8027-1}.
\newblock URL \url{http://dx.doi.org/10.1007/s229-001-8027-1}.

\bibitem[Bruinier and Ono(2013)]{BO}
Jan Bruinier and Ken Ono.
\newblock Algebraic formulas for the coefficients of half-integral weight
  harmonic weak {M}aass forms.
\newblock \emph{Adv. Math.}, 246:\penalty0 198--219, 2013.
\newblock ISSN 0001-8708.
\newblock \doi{10.1016/j.aim.2013.05.028}.
\newblock URL \url{https://dx.doi.org/10.1016/j.aim.2013.05.028}.

\bibitem[Cowan et~al.(2017)Cowan, Katz, and White]{CKW}
Raemeon Cowan, Daniel Katz, and Lauren White.
\newblock A new generating function for calculating the {I}gusa local zeta
  function.
\newblock \emph{Adv. Math.}, 304:\penalty0 355--420, 2017.
\newblock ISSN 0001-8708.
\newblock \doi{10.1016/j.aim.2016.09.003}.
\newblock URL \url{http://dx.doi.org/10.1016/j.aim.2016.09.003}.

\bibitem[Diamond and Shurman(2005)]{DS}
Fred Diamond and Jerry Shurman.
\newblock \emph{A first course in modular forms}, volume 228 of \emph{Graduate
  Texts in Mathematics}.
\newblock Springer-Verlag, New York, 2005.
\newblock ISBN 0-387-23229-X.

\bibitem[Eichler and Zagier(1985)]{EZ}
Martin Eichler and Don Zagier.
\newblock \emph{The theory of {J}acobi forms}, volume~55 of \emph{Progress in
  Mathematics}.
\newblock Birkh\"auser Boston, Inc., Boston, MA, 1985.
\newblock ISBN 0-8176-3180-1.
\newblock \doi{10.1007/978-1-4684-9162-3}.
\newblock URL \url{http://dx.doi.org/10.1007/978-1-4684-9162-3}.

\bibitem[Gross and Zagier(1986)]{GZ}
Benedict Gross and Don Zagier.
\newblock Heegner points and derivatives of {$L$}-series.
\newblock \emph{Invent. Math.}, 84\penalty0 (2):\penalty0 225--320, 1986.
\newblock ISSN 0020-9910.
\newblock \doi{10.1007/BF01388809}.
\newblock URL \url{http://dx.doi.org/10.1007/BF01388809}.

\bibitem[Hirzebruch and Zagier(1976)]{HZ}
Friedrich Hirzebruch and Don Zagier.
\newblock Intersection numbers of curves on {H}ilbert modular surfaces and
  modular forms of {N}ebentypus.
\newblock \emph{Invent. Math.}, 36:\penalty0 57--113, 1976.
\newblock ISSN 0020-9910.
\newblock \doi{10.1007/BF01390005}.
\newblock URL \url{http://dx.doi.org/10.1007/BF01390005}.

\bibitem[Ono(2009)]{O}
Ken Ono.
\newblock Unearthing the visions of a master: harmonic {M}aass forms and number
  theory.
\newblock In \emph{Current developments in mathematics, 2008}, pages 347--454.
  Int. Press, Somerville, MA, 2009.

\bibitem[Skoruppa(1985)]{SK}
Nils-Peter Skoruppa.
\newblock \emph{\"Uber den {Z}usammenhang zwischen {J}acobiformen und
  {M}odulformen halbganzen {G}ewichts}, volume 159 of \emph{Bonner
  Mathematische Schriften [Bonn Mathematical Publications]}.
\newblock Universit\"at Bonn, Mathematisches Institut, Bonn, 1985.
\newblock Dissertation, Rheinische Friedrich-Wilhelms-Universit\"at, Bonn,
  1984.

\bibitem[Washington(1997)]{Wa}
Lawrence Washington.
\newblock \emph{Introduction to cyclotomic fields}, volume~83 of \emph{Graduate
  Texts in Mathematics}.
\newblock Springer-Verlag, New York, second edition, 1997.
\newblock ISBN 0-387-94762-0.
\newblock \doi{10.1007/978-1-4612-1934-7}.
\newblock URL \url{http://dx.doi.org/10.1007/978-1-4612-1934-7}.

\bibitem[Williams()]{W}
Brandon Williams.
\newblock Poincar\'e square series for the {W}eil representation.
\newblock Preprint.
\newblock URL \url{https://arxiv.org/abs/1704.06758}.

\end{thebibliography}

\end{document}